\documentclass[11pt,a4paper]{article}
\usepackage{fancyhdr}
\usepackage{amsmath}
\usepackage{amsthm}
\usepackage{amssymb}
\usepackage{wasysym}
\usepackage{paralist}
\usepackage{makeidx}
\usepackage[all]{xy}
\usepackage{mathdots}
\usepackage{yhmath}
\usepackage{mathtools}

\usepackage{young}
\usepackage[vcentermath]{youngtab}

\usepackage{graphicx}
\usepackage{epstopdf}

\input xy

\newcommand{\nc}{\newcommand}
 
 \nc{\cl}{\centerline}
 \renewcommand{\l}{{\rm len}}
 \nc{\SL}{{\rm SL}}
 \nc{\hatQ}{{\hat Q}}
 \nc{\sgn}{{\rm sgn}}
 
 \nc{\seee}{\mathbb C}
 
 \newcommand{\id}{{\rm id}}
 
  \newcommand{\stind}{{\rm stind}}
 
 \nc{\hatlambda}{{\hat\lambda}}
 
 \nc{\daggerlambda}{{\lambda^\dagger}}

  \newcommand{\dotG}{{\dot{G}}}
    \newcommand{\dotL}{{\dot{L}}}
  
  \newcommand{\dotnabla}{{\dot{\nabla}}}
  
    \newcommand{\dotI}{{\dot{I}}}

  \newcommand{\resp}{{resp.\,}}

  \newcommand{\Hec}{{\rm Hec}}

\nc\diag{{\rm diag}}
\renewcommand{\vert}{{\,|\,}}
\nc{\hatL}{{\hat L}}

\nc{\barE}{{\bar   E}}
\nc{\D}{{\mathcal D}}
\nc{\E}{{\mathcal E}}
\nc{\F}{{\mathcal F}}
\nc{\FF}{{\mathcal F}}
\nc{\I}{{\mathcal I}}
\nc{\even}{{\rm e}}
\nc{\ep}{\epsilon}
\nc{\odd}{{\rm o}}
\nc{\Coker}{{\rm Coker}}
\nc{\olE}{{\overline E}}
\nc{\indBG}{{\rm ind}_B^G\,}
\nc{\indHG}{{\rm ind}_H^G\,}

\nc{\que}{{\mathbb Q}}
\nc{\barlambda}{{\bar\lambda}}
\nc{\barmu}{{\bar\mu}}
\nc{\barnu}{{\bar\nu}}
\nc{\bartau}{{\bar\tau}}
\nc{\barm}{{\bar m}}
\nc{\divind}{{\rm div.ind}}
\nc{\tl}{{\tilde{\lambda}}}
\nc{\dar}{\downarrow}

\nc{\Sym}{{\rm \Sigma}}
\nc{\Symm}{{\rm Sym}}

\newcommand{\q}{\quad}
\newcommand{\de}{\delta}

\newcommand{\Mod}{{\rm Mod}}
\renewcommand{\mod}{{\rm mod}}

\newcommand{\Sp}{{\rm Sp}}

\newcommand{\bs}{\bigskip}

\renewcommand{\vert}{\,|\,}

\renewcommand{\sgn}{{\rm sgn}}

\newcommand{\pol}{{\rm pol}}

\xyoption{all}

\newcommand{\ind}{{\rm ind}}
\renewcommand{\vert}{\,|\,}
 
\newcommand{\zed}{{\mathbb Z}}
\newcommand{\Ext}{{\rm Ext}}

\newcommand{\Hom}{{\rm Hom}}
\newcommand{\cf}{{\rm cf}}

\renewcommand{\mod}{{\rm mod}}

\newcommand{\GL}{{\rm GL}}

\renewcommand{\mod}{{\rm{mod}}}

\nc{\geom}{{\rm geom}}
\nc{\rep}{{\rm rep}}

\newtheorem{definition}{Definition}[section]
\newtheorem{proposition}[definition]{Proposition}
\newtheorem{theorem}[definition]{Theorem}
\newtheorem{lemma}[definition]{Lemma}
\newtheorem{corollary}[definition]{Corollary}
\newtheorem{example}[definition]{Example}
\newtheorem{examples}[definition]{Examples}
\newtheorem{remark}[definition]{Remark}

\begin{document}


\centerline{\bf Injective Schur Modules}

\bigskip

\centerline{Stephen Donkin  and Haralampos Geranios}

\bigskip

{\it Department of Mathematics, University of York, York YO10 5DD}\\

\medskip

{\tt stephen.donkin@york.ac.uk,  haralampos.geranios@york.ac.uk}

\bs

\centerline{16  March 2017}
\bs\bs\bs

\section*{Abstract}

\q  We determine the  partitions $\lambda$ for which  the  corresponding induced module (or Schur module in the language of Buchsbaum et. al., \cite{ABW}) $\nabla(\lambda)$ is injective in the category of polynomial modules for a general linear group over an infinite field,  equivalently  which Weyl modules are projective polynomial modules. Since the problem is essentially no more difficult in the quantised case we address it at this level of generality.  Expressing our results in terms of the representation theory of Hecke algebras at the parameter $q$  we determine the  partitions $\lambda$ for which  the corresponding Specht module is a Young module,  when $1+q\neq 0$. In the classical case this problem was addressed by D. Hemmer, \cite{Hem}.   The nature of the set of partitions appearing in our solution gives a new formulation of Carter's condition on regular partitions. On the other hand, we note, in Remark 2.22,  that the result on irreducible Weyl modules for the quantised Schur algebra  $S_q(n,n)$,  \cite{MathasBook}, Theorem 5.39, given in terms of Carter partitions, may be also used to obtain the main result presented here.

\section*{Introduction}

\q   Let $K$ be a field and  $0\neq q\in K$. Let  $G(n)$ be the corresponding quantum general linear   of degree $n$, as in, for example \cite{D3}.  Let $T(n)$ denote the algebraic torus and $B(n)$  the Borel subgroup, as in \cite{D3}. Each partition $\lambda$ with at most $n$ parts determines a one dimensional $T(n)$-module $K_\lambda$  and the module structure  extends uniquely to give the structure of a  $B(n)$-module. The irreducible polynomial representations of $G(n)$ parametrised by partitions with at most $n$ parts and we write $L_n(\lambda)$ for the irreducible module corresponding to the partition $\lambda$. We set $\nabla_n(\lambda)$ to be the induced module $\ind_{B(n)}^{G(n)}  K_\lambda$. This module contains a unique copy of the irreducible module $L_n(\lambda)$. We write $I_n(\lambda)$ for the injective envelope of  $L_n(\lambda)$ and we have an embedding of  $\nabla_n(\lambda)$ into $I_n(\lambda)$. In this paper we give a combinatorial description of those $\lambda$ such that this embedding is an isomorphism and so the induced module $\nabla_n(\lambda)$ is injective as a polynomial module. Also, considering the contravariant duals of these modules, in the sense of \cite{EGS} and \cite{D3}, we describe which Weyl modules $\Delta_n(\lambda)$ are projective as polynomial modules of $G(n)$.  In the last section of this paper we express our results in terms of the representation theory of Hecke algebras and we determine the partitions $\lambda$ for which the corresponding Specht module $\Sp(\lambda)$ is a Young module,  in case $q\neq -1$. For another approach for the latter result in the classical case see  \cite{Hem}.

\q We deal with the preliminary material in Section 1.  In Section 2 we give an explicit description of the partitions satisfying the above injectivity condition. We then  show that these partitions are exactly those satisfying Carter's condition. 
In Section 3, we relate the forgoing material to results on  Specht modules and Young modules for Hecke algebras.

\bs\bs\bs


\section{Preliminaries}

\subsection{Combinatorics}

\q The  standard  reference for  the polynomial representation theory of  \\
$\GL_n(K)$ is the monograph \cite{EGS}.   Though we work in the quantised context this reference is appropriate  as the combinatorial set-up  is  essentially the same and we adopt the notation of \cite{EGS} wherever  convenient.  Further details may also be found in the monograph \cite{D3}, which treats the quantised case.

\q By a partition we mean an infinite  sequence $\lambda=(\lambda_1,\lambda_2,\ldots)$ of nonnegative integers with $\lambda_1\geq\lambda_2\geq \ldots$ and $\lambda_j=0$ for all $j$ sufficiently large.   If $m$ is a positive integer such that $\lambda_j=0$ for $j>m$ we identify $\lambda$ with the finite sequence $(\lambda_1,\ldots,\lambda_m)$.  The length   $\l(\lambda)$ of  a partition $\lambda=(\lambda_1,\lambda_2,\ldots)$ is $0$ if $\lambda=0$ and  is the positive integer $m$ such that  $\lambda_m\neq 0$, $\lambda_{m+1}=0$, if $\lambda\neq 0$. For a partition $\lambda$, we denote  by $\lambda'$ the transpose partition of $\lambda$.  We define the   degree of a partition $\lambda=(\lambda_1,\lambda_2,\ldots)$ by $\deg(\lambda)=\lambda_1+\lambda_2+\cdots$.

\q We set $X(n)=\zed^n$. There is a natural partial order on $X(n)$. For $\lambda=(\lambda_1,\ldots,\lambda_n), \mu=(\mu_1,\ldots,\mu_n)\in X(n)$,  we write $\lambda\leq \mu$ if $\lambda_1+\cdots+\lambda_i\leq \mu_1+\cdots+\mu_i$ for $i=1,2,\ldots,n-1$ and $\lambda_1+\cdots+\lambda_n=\mu_1+\cdots+\mu_n$. We shall use the standard $\zed$-basis   $\ep_1,\ldots,\ep_n$ of   $X(n)$. Thus  $\ep_i=(0,\ldots,1,\ldots,0)$ (with $1$ in the $i$th position), for $1\leq i\leq n$.     We  write $\omega_i$ for the element $\ep_1+\cdots+\ep_i$ of $X(n)$, for $1\leq i\leq n$, and set $\omega_0=0$.  We write $\Lambda(n)$ for the set of $n$-tuples of nonnegative integers.   We write $X^+(n)$ for the set of dominant $n$-tuples of integers, i.e., the set of elements $\lambda=(\lambda_1,\ldots,\lambda_n)\in X(n)$ such that $\lambda_1\geq \cdots\geq  \lambda_n$. 

\q We write  $\Lambda^+(n)$ for the set of partitions into at most $n$-parts, i.e.,  $\Lambda^+(n)=X^+(n)\bigcap \Lambda(n)$.   For a nonnegative integer $r$ we write $\Lambda^+(n,r)$ for the set of partitions of $r$ into at most $n$ parts, i.e., the set of elements of $\Lambda^+(n)$ of degree $r$.

 

 \subsection{Rational and Polynomial Modules}

 \q Appropriate references for the set-up described here are \cite{DD}, \cite{DStd}, \cite{D3}.  Let $K$ be a field. If $V,W$ are vector spaces over $K$, we write $V\otimes W$ for the tensor product $V\otimes_K W$.   We shall be working with the representation theory of quantum groups over $K$. By the category of quantum groups over $K$ we understand the opposite category of the category of Hopf algebras over $K$. Less formally we shall use the expression \lq\lq $G$ is a quantum group" to indicate that we have in mind a Hopf algebra over $K$ which we denote $K[G]$ and call the coordinate algebra of $G$.  We say that $\phi:G\to H$ is a morphism of quantum groups over $K$ to indicate that we have in mind a morphism of Hopf algebras over $K$, from $K[H]$ to $K[G]$, denoted $\phi^\sharp$ and called the co-morphism of $\phi$. We will say $H$ is a quantum subgroup of the quantum group $G$, over $K$, to indicate that $H$ is a quantum group with coordinate algebra $K[H]=K[G]/I$, for some Hopf ideal $I$ of $K[G]$, which we call the defining ideal of $H$.  The inclusion morphism $i:H\to G$ is the morphism of quantum groups whose co-morphism $i^\sharp:K[G]\to K[H]=K[G]/I$ is the natural map. 

\q Let $G$ be a quantum group over $K$. The category of  left (\resp  right) $G$-modules is the the category of right (\resp  left) $K[G]$-comodules.  We write $\Mod(G)$ for the category of left $G$-modules and $\mod(G)$ for the category of finite dimensional left $G$-modules.  We shall also call a $G$-module a rational $G$-module (by analogy with the representation theory of algebraic groups).  A $G$-module will mean a left $G$-module unless  indicated otherwise. For a finite dimensional $G$-module $V$ the dual space $V^*=\Hom_K(V,K)$ has a natural $G$-module structure.
For a finite dimensional $G$-module $V$ and a non-negative integer $r$ we write $V^{\otimes r}$ for the $r$-fold tensor product $V\otimes V\otimes \cdots \otimes V$ and write $V^{\otimes -r}$ for the dual of $V^{\otimes r}$.

 \q Let $V$ be a finite dimensional $G$-module with structure map $\tau: V\to V\otimes K[G]$.  The coefficient space $\cf(V)$ of $V$ is the subspace of $K[G]$ spanned by the \lq\lq coefficient elements" $f_{ij}$, $1\leq i,j\leq m$, defined with respect to a basis $v_1,\ldots, v_m$ of $V$, by the equations 
 $$\tau(v_i)=\sum_{j=1}^m v_j\otimes f_{ji}$$
 for $1\leq i\leq m$.  The coefficient space $\cf(V)$ is independent of the choice of basis and is a subcoalgebra of $K[G]$.

\q We fix $0\neq q\in K$. For a positive integer $n$  we shall be working with the corresponding quantum general linear group $G(n)$, as in \cite{D3}.   
We have a $K$-bialgebra $A(n)$ given by generators $c_{ij}$, $1\leq i,j\leq n$, subject to certain relations (depending on $q$), as in \cite{D3}, 0.20.  The comultiplication map  $\delta:A(n)\to A(n)\otimes A(n)$ satisfies $\de(c_{ij})=\sum_{r=1}^n c_{ir}\otimes c_{rj}$ and the augmentation map $\ep:A(n)\to K$ satisfies $\ep(c_{ij})=\de_{ij}$ (the Kronecker delta), for $1\leq i,j\leq n$.  The elements $c_{ij}$ will be called the coordinate elements and we define the determinant element
$$d_n=\sum_{\pi\in \Symm(n)} \sgn(\pi) c_{1,\pi(1)}\ldots c_{n,\pi(n)},$$
where $\sgn(\pi)$ denotes the sign of the permutation $\pi$. We form the Ore localisation $A(n)_{d_n}$. The comultiplication map $A(n)\to A(n)\otimes A(n)$ and augmentation map $A(n)\to K$ extend uniquely to $K$-algebraic maps $A(n)_{d_n}\to A(n)_{d_n}\otimes A(n)_{d_n}$ and $A(n)_{d_n}\to K$, giving $A(n)_{d_n}$ the structure of a Hopf algebra. By the quantum general linear group $G(n)$ we mean the quantum group over $K$ with coordinate algebra $K[G(n)]=A(n)_{d_n}$.

\q We write $T(n)$ for the quantum subgroup of $G(n)$ with defining ideal generated by all $c_{ij}$ with $1\leq i,j\leq n$, $i\neq j$.  We write $B(n)$ for quantum subgroup of $G(n)$ with defining ideal generated by all $c_{ij}$ with $1\leq i <j\leq n$. We call $T(n)$ a maximal torus and $B(n)$ a Borel subgroup of $G(n)$ (by analogy with the classical case).

\q We now recall the weight space decomposition of a finite dimensional  $T(n)$-module.  For $1\leq i\leq n$ we define ${\bar c}_{ii}=c_{ii}+I_{T(n)}\in K[T(n)]$, where $I_{T(n)}$ is the defining ideal of the quantum subgroup $T(n)$ of $G(n)$.  Note that ${\bar c}_{11}\ldots {\bar c}_{nn}=d_n+I_{T(n)}$, in particular each ${\bar c}_{ii}$ is invertible in $K[T(n)]$. For $\lambda=(\lambda_1,\ldots,\lambda_n)\in X(n)$ we define ${\bar c}^\lambda={\bar c}_{11}^{\lambda_1}\ldots {\bar c}_{nn}^{\lambda_n}$. The elements ${\bar c}^\lambda$, $\lambda\in X(n)$, are group-like and form a $K$-basis of $K[T(n)]$. 
For $\lambda=(\lambda_1,\ldots,\lambda_n)\in X(n)$, we write $K_\lambda$ for $K$ regarded as a (one dimensional) $T(n)$-module with structure map $\tau:K_\lambda\to K_\lambda\otimes K[T(n)]$ given by $\tau(v)=v\otimes {\bar c}^\lambda$, $v\in K_\lambda$.  For a finite dimensional  rational $T(n)$-module $V$ with structure map $\tau:V\to V\otimes K[T(n)]$  and $\lambda\in X(n)$ we have the weight space 
$$V^\lambda=\{v\in V\vert \tau(v) =v\otimes {\bar c}^\lambda\}.$$  
Moreover, we have the weight space decomposition $V=\bigoplus_{\lambda\in X(n)} V^\lambda$. 
We say that $\lambda\in X(n)$ is a weight of $V$ if $V^\lambda\neq 0$.

\q For each $\lambda\in X^+(n)$ there is an irreducible rational $G(n)$-module $L_n(\lambda)$ which has unique highest weight $\lambda$ and such $\lambda$ occurs as a weight with multiplicity one. The modules $L_n(\lambda)$, $\lambda\in X^+(n)$, form a complete set of pairwise non-isomorphic irreducible rational  $G(n)$-modules. 
 Note that for $\lambda=(\lambda_1,\ldots,\lambda_n)\in X^+(n)$ the dual module $L_n(\lambda)^*$ is isomorphic to $L_n(\lambda^*)$, where $\lambda^*=(-\lambda_n,\ldots,-\lambda_1)$. For a finite dimensional rational $G(n)$-module $V$ and $\lambda\in X^+(n)$ we write $[V:L_n(\lambda)]$ for the multiplicity of $L_n(\lambda)$ as a composition factor of $V$. 

\q We write $D_n$ for the one dimensional $G(n)$-module corresponding to the determinant. Thus $D_n$ has structure map $\tau:D_n\to D_n\otimes K[G]$, given by $\tau(v)=v\otimes d_n$, for $v\in D_n$. 
Thus we have  $D_n=L_n(\omega)=L_n(1,1,\ldots,1)$. We write $E_n$ for the natural $G(n)$-module. Thus  $E_n$ has basis $e_1,\ldots,e_n$,  and the structure map $\tau:E_n\to   E_n\otimes K[G(n)]$ is given by $\tau(e_i)=\sum_{j=1}^n e_j\otimes c_{ji}$.

\q A finite dimensional $G(n)$-module $V$ is called polynomial if $\cf(V) \leq A(n)$. The modules $L_n(\lambda)$, $\lambda\in \Lambda^+(n)$, form a complete set of pairwise non-isomorphic irreducible polynomial $G(n)$-modules. We write $I_n(\lambda)$ for the injective envelope of $L_n(\lambda)$ in the category of polynomial modules.  We have a grading   $A(n)=\bigoplus_{r=0}^\infty  A(n,r)$ in  such a way that each $c_{ij}$ has degree $1$. Moreover each $A(n,r)$ is a finite dimensional subcoalgebra of $A(n)$. The dual algebra $S(n,r)$ is known as the $q$-Schur algebra.  A finite dimensional $G(n)$-module $V$ is polynomial of degree $r$ if $\cf(V)\leq A(n,r)$.  We write $\pol(n)$ (\resp  $\pol(n,r)$)  for the full subcategory of $\mod(G(n))$ whose objects are the polynomial modules (\resp  the modules which are polynomial of degree $r$).

\q For an arbitrary finite dimensional polynomial $G(n)$-module we may write $V$ uniquely as a direct sum $V=\bigoplus_{r=0}^\infty V(r)$ in such a way that $V(r)$ is polynomial of degree $r$, for $r\geq 0$. The modules $L_n(\lambda)$, $\lambda\in\Lambda^+(n,r)$, form a complete set of pairwise non-isomorphic irreducible polynomial $G(n)$-modules which are polynomial of degree $r$.  We write $\mod(S)$ for the category of left modules for a finite dimensional $K$-algebra $S$.  The category 
 $\pol(n,r)$ is naturally equivalent to the category $\mod(S(n,r))$.    It follows in particular  that, for $\lambda\in \Lambda^+(n,r)$, the module $I_n(\lambda)$ is a finite dimensional module which is polynomial of degree $r$.

\q We shall also need modules induced from $B(n)$ to $G(n)$.  (For details of the induction functor $\Mod(B(n))\to \Mod(G(n))$ see, for example, \cite{DStd}).  For $\lambda\in X(n)$ there is a unique (up to isomorphism) one dimensional $B(n)$-module whose restriction to $T(n)$ is  $K_\lambda$. We also denote this module by $K_\lambda$. The induced module $\ind_{B(n)}^{G(n)} K_\lambda$ is non-zero if and only if $\lambda\in X^+(n)$. For $\lambda\in X^+(n)$ we set $\nabla_n(\lambda)=\ind_{B(n)}^{G(n)}  K_\lambda$. Then $\nabla_n(\lambda)$ is finite dimensional  (and its character is the Schur symmetric function corresponding to $\lambda$).  The $G(n)$-module socle of $\nabla_n(\lambda)$ is $L_n(\lambda)$. The module $\nabla_n(\lambda)$ has unique highest weight $\lambda$ and this weight occurs with multiplicity one.

\q A filtration $0=V_0\leq V_1\leq \cdots\leq V_r=V$ of  a finite dimensional rational $G(n)$-module $V$ is said to be {\it{good}} if for each $1\leq i\leq r$ the quotient $V_i/V_{i-1}$ is either zero or isomorphic to $\nabla_n(\lambda^i)$ for some $\lambda^i\in X^+(n)$.  For a rational $G(n)$-module $V$ admitting a good filtration, for each $\lambda\in X^+(n)$ the multiplicity 
$|\{1\leq i\leq r\vert V_i/V_{i-1}\cong \nabla_n(\lambda)\}|$ is independent of the choice of the good filtration, and will be denoted $(V:\nabla_n(\lambda))$. 



\q For a partition $\lambda$ we denote by $[\lambda]$ the corresponding partition diagram (as in \cite{EGS}). For a positive integer $l$, the $l$-core of $[\lambda]$ is the diagram obtained by removing skew $l$-hooks,  as in \cite{James}. If $\lambda,\mu\in \Lambda^+(n,r)$ and $[\lambda]$ and $[\mu]$ have different $l$-cores then the simple modules $L_n(\lambda)$ and $L_n(\mu)$ belong to different blocks and we have in particular  $\Ext^i_{S(n,r)}(\nabla(\lambda),\nabla(\mu))=0$, for all $i\geq 0$.  A precise  description of the blocks of the $q$-Schur algebras was found by Cox, see \cite{Cox},  Theorem 5.3.

\q  For $\lambda\in \Lambda^+(n)$  the module $I_n(\lambda)$ has a good filtration and we have the reciprocity formula $(I_n(\lambda):\nabla_n(\mu))=[\nabla_n(\mu):L_n(\lambda)]$ see e.g., \cite{DStd}, Section 4, (6). 

\q For $\lambda\in \Lambda^+(n,r)$ we write $\Delta_n(\lambda)$ for the Weyl module corresponding to the partition $\lambda$. Then  $\Delta_n(\lambda)$ is the contravariant dual of the induced module $\nabla_n(\lambda)$. We note by $P_n(\lambda)$ the projective cover of $L_n(\lambda)$ in the category of polynomial modules. The module  $P_n(\lambda)$ is the contravariant dual of $I_n(\lambda)$. 

\q If $q$ is not a root of unity or if $q=1$ and $K$ has characteristic $0$ then all $G(n)$-modules are semisimple, see e.g.,  \cite{DD}, (3.3.2) or \cite{DStd},  Section 4, (8) and so all the polynomial modules are polynomially injective. So the problem addressed in this paper is trivial in these  cases and we shall assume from now on  that $q$ is a root of unity and if $q=1$ then $K$ has positive characteristic.

\q Let $l$ be the   smallest positive integer such that $1+q+\cdots+q^{l-1}=0$.  Thus $l$ is the order of $q$ if $q\neq 1$ and $l$ is the characteristic  of $K$ if $q=1$ and $K$ has positive characteristic.

\subsection{Connections with the Hecke algebras}

\q We now record some connections with representations of  Hecke algebra of type $A$. We fix a positive integer $r$.  We write $\l(\pi)$ for the length of a permutation $\pi$.   The Hecke algebra $\Hec(r)$ is the $K$-algebra with basis $T_w$, $w\in \Symm(r)$, and multiplication satisfying

\begin{align*}&T_wT_{w'}=T_{ww'}, \quad \hbox{ if } \l(ww')=\l(w)+\l(w'), \hbox{and}\cr
&(T_s+1)(T_s-q)=0
\end{align*}
for $w,w'\in \Symm(r)$ and a basic transposition $s\in \Symm(r)$. For brevity  we will denote the Hecke algebra $\Hec(r)$ by $H(r)$.

\q Assume now $n\geq r$.  We have the Schur functor $f:\mod(S(n,r))\to \mod(H(r))$, see \cite{D3}, 2.1.  For $\lambda$ a partition of degree $r$ we denote  by $\Sp(\lambda)$ the corresponding (Dipper-James) Specht module and by $Y(\lambda)$ the corresponding Young module of $H(r)$. By \cite{D3}, Sections 4.4 and 4.5 we have  the following results.

\bs

{\bf Proposition 1.3.1}  The functor $f$ has the following properties :\\
(i) $f$ is exact;\\
(ii) for $\lambda\in \Lambda^+(n,r)$ we have $f\nabla_n(\lambda)=\Sp(\lambda)$;\\
(iii) for $\lambda\in \Lambda^+(n,r)$ we have $fI_n(\lambda)=Y(\lambda)$;\\

\bs

We will need some further  connections between the representations of the Schur algebras and the Hecke algebras. By \cite{DJ} Corollary 8.6 we have the following. 

\bs

{\bf Proposition 1.3.2}  \,    Let $l\geq 3$. For $\lambda,\mu\in\Lambda^+(n,r)$ we have:\\
(i) $\Hom_{S(n,r)}(\nabla_n(\lambda),\nabla_n(\mu))=\Hom_{H(r)}(\Sp(\lambda),\Sp(\mu))$;\\
(ii) If $\Hom_{H(r)}(\Sp(\lambda),\Sp(\mu))\neq0$ then $\lambda\geq \mu$.

\bs

\q A filtration $0=V_0\leq V_1\leq \cdots\leq V_s=V$ of  a finite dimensional $H(r)$-module $V$ is called a {\it{Specht filtration}}  if for each $1\leq i\leq r$ the quotient $V_i/V_{i-1}$ is isomorphic to $\Sp(\lambda^i)$ for some partition $\lambda^i$ of degree $r$. By \cite{D1} Proposition 10.6 or alternatively by \cite{HN} Theorem 3.7.1  we have that if $l\geq4$ and $V$ is an $H(r)$-module admitting a Specht filtration, then for each $\lambda$ of degree $r$ the multiplicity 
$|\{1\leq i\leq s\vert V_i/V_{i-1}\cong \Sp(\lambda))\}|$ is independent of the choice of the Specht filtration, i.e. these multiplicities are well defined.

\bs

 \q Let $\alpha\in \Lambda(n,r)$. We write $H(\alpha)$ for the subalgebra  $H(\alpha_1) \otimes\dots\otimes H(\alpha_n)$ of $H(r)$.  By \cite{D3} Section 4.4, (1), (iii) and  (3),(ii) and by  \cite{D1} Proposition 10.6, we have the following results. 

\bs 

{\bf Proposition 1.3.3}\\
(i) The modules $\{Y(\lambda)\mid \lambda\in\Lambda^+(n,r)\}$ are pairwise non-isomorphic and are precisely (up to isomorphism) the indecomposable summands of  the modules $H(r)\otimes_{H(\alpha)}k$;\\
(ii) Let $\lambda\in \Lambda^+(n,r)$. Then $Y(\lambda)$ has a Specht filtration $0=Y_0<Y_1<\dots<Y_s=Y(\lambda)$ with sections $Y^i/Y^{i-1}\cong\Sp(\mu^i)$ with $\mu^i\in \Lambda^+(n,r)$  for $1\leq i\leq s$ and $\mu^j<\mu^i$ implies that $j<i$  for $1\leq i,j\leq s$. Moreover  for each $\mu\in \Lambda^+(n,r)$ we have $|\{i\in[1,s]\mid\mu^i=\mu\}|=(I_n(\lambda):\nabla_n(\mu))$.

\bs\bs\bs


\section{Injective Partitions}

 \q   We write $X_1(n)$ for the set of $l$-restricted partition into at most $n$ parts, i.e., the set of elements $\lambda=(\lambda_1,\ldots,\lambda_n)\in \Lambda^+(n)$ such that $0\leq \lambda_1-\lambda_2,\ldots,\lambda_{n-1}-\lambda_n, \lambda_n<l$.

\q Let $\lambda\in \Lambda^+(n)$.   Recall that the induced module $\nabla_n(\lambda)$ has simple socle $L_n(\lambda)$, so that $\nabla_n(\lambda)$ embeds in   $I_n(\lambda)$. We are interested in the cases in which this embedding is an isomorphism.

\begin{definition}   We call an element $\lambda$ of $\Lambda^+(n)$ an  injective partition for $G(n)$,   or just an  injective partition relative to $n$, if  $\nabla_n(\lambda)$ is injective in the category of polynomial $G(n)$-modules, i.e., if  $\nabla_n(\lambda)=I_n(\lambda)$.

\end{definition}

\q Let $\lambda,\mu\in \Lambda^+(n,r)$.  We may also consider $\lambda$ and $\mu$ as elements of $\Lambda^+(N)$ for $N\geq n$ and we have  $[\nabla_n(\lambda):L_n(\mu)]=[\nabla_N(\lambda):L_N(\mu)]$, by \cite{D3}, 4.2, (6)  (see \cite{EGS}, (6.6e) Theorem for the classical case). We shall  write simply $[\lambda:\mu]$ for 
$[\nabla_n(\lambda):L_n(\mu)]$.

\begin{remark}   Let  $\lambda\in \Lambda^+(n)$ and suppose $\lambda$ has degree $r$.  For  $\mu\in \Lambda^+(n,r)$ we have  $(I_n(\lambda):\nabla_n(\mu))=[\mu:\lambda]$. In particular we have $(I_n(\lambda):\nabla_n(\lambda))=1$ and if $(I_n(\lambda):\nabla_n(\mu))\neq 0$ then $\mu\geq \lambda$. Thus $\lambda$ is injective for $G(n)$ if and only if $[\mu:\lambda]=0$ for all $\mu\in \Lambda^+(n,r)$ with  $\mu>\lambda$.  

\q Suppose  $\lambda$ is  injective for $G(n)$ and $N\geq n$. Let $\mu\in \Lambda^+(N,r)$ and suppose $\mu>\lambda$. Then $\mu$ has at most $n$ parts, i.e., $\mu\in \Lambda^+(n,r)$,  and therefore $[\mu:\lambda|=0$. Thus if $\lambda$ is injective for $G(n)$  then it is injective for $G(N)$  for all $N\geq n$. 
\end{remark}

\q From now on we shall simply say that a partition $\lambda$ is injective if it is injective for some, and hence every,   $G(n)$  with  $n\geq {\rm len}(\lambda)$.

\q Henceforth, for a partition  $\lambda$,  we write simply $\nabla(\lambda)$ for $\nabla_n(\lambda)$, write $L(\lambda)$ for $L_n(\lambda)$  and so on, with  $n$ understood to be sufficiently large, where confusion seems unlikely.

\begin{lemma} If $\lambda$ is injective and $n={\rm len}(\lambda)$ then $\lambda-\omega_n$ is injective.

\end{lemma}

\begin{proof}  We work with $G(n)$-modules.  Suppose that $\mu$ is a partition greater than $\lambda-\omega_n$ in the dominance order.  We  have 
$$[\nabla_n(\mu):L_n(\lambda-\omega_n)]=[D_n\otimes \nabla_n(\mu):D_n\otimes L_n(\lambda-\omega_n)]=[\mu+\omega_n:\lambda]$$
and this is $0$ since $\mu+\omega_n>\lambda$. Hence $\lambda-\omega_n$ is injective by Remark 2.2.

\end{proof}

\begin{lemma}  A partition $\lambda$ is injective if and only if $\lambda$ is a maximal weight of $I(\lambda)$.

\end{lemma}

\begin{proof}   The module $\nabla(\lambda)$ has maximal weight $\lambda$ so if $\lambda$ is injective it is a maximal weight of $I(\lambda)$.

\q Suppose conversely that $\lambda$ is a maximal weight of $I(\lambda)$. Let $\mu\in \Lambda^+(n,r)$ with $(I(\lambda):\nabla(\mu))\neq 0$ and hence, by reciprocity, $[\mu:\lambda]\neq 0$. Then $\mu\geq\lambda$ and by maximality $\mu=\lambda$ and so $\lambda$ is injective, by Remark 2.2. 
\end{proof}

\q Given a partition $\lambda$ we may write $\lambda$ uniquely in the form $\lambda=\lambda^0+l\barlambda$, where $\lambda^0,\barlambda$ are partitions and $\lambda^0$ is $l$-restricted.

\q \q It will be important for us to make a comparison with the classical case $q=1$. In this case we will write $\dotG(n)$ for  $G(n)$ and write $x_{ij}$ for the coordinate element $c_{ij}$, $1\leq i,j\leq n$. We also write $\dotL_n(\lambda)$ for the $\dotG(n)$-module $L_n(\lambda)$, $\lambda\in X^+(n)$.

\q Now we have a morphism of Hopf algebras $\theta: K[\dotG(n)]\to K[G(n)]$ given by $\theta(x_{ij})=c_{ij}^l$, for $1\leq i,j\leq n$.  We write $F:G(n)\to \dotG(n)$ for the morphism of quantum groups such that $F^\sharp=\theta$.  Given a $\dotG(n)$-module $V$ we write $V^F$ for the corresponding $G(n)$-module. Thus, $V^F$ as a vector space is $V$ and if the $\dotG(n)$-module $V$ has structure map $\tau:V\to V\otimes K[\dotG(n)]$ then $V^F$ has structure map $(\id_V \otimes F)\circ \tau: V^F \to  V^F\otimes K[G(n)]$, where $\id_V:V\to V$ is the identity map on the vector space  $V$.  


\q  We have the following relationship between the irreducible modules for $G(n)$ and ${\dotG(n)}$, see \cite{D3}, Section 3.2, (5).

\begin{theorem} 
(Steinberg's Tensor Product Theorem)   For  $\lambda^0\in X_1(n)$ and $\barlambda\in X^+(n)$ we have 
$$L_n(\lambda^0+l\barlambda)\cong L_n(\lambda^0)\otimes \dotL_n(\barlambda)^F.$$

\end{theorem}

\begin{lemma} If $\lambda$ is an injective partition for $G(n)$ then $\lambda^0$ is injective for $G(n)$ and $\barlambda$ is injective for $\dotG(n)$.

\end{lemma}

\begin{proof}    We write $G_1$ for the first infinitesimal subgroup of $G(n)$.  The  $G_1$-socle of $\nabla(\lambda)$ is $L(\lambda^0)\otimes \dotnabla(\barlambda)^F$, and the $G_1$-socle of  $I(\lambda)$ is $L(\lambda^0)\otimes \dotI(\barlambda)^F$ by \cite{DG1},  Lemma 3.2 (i)   (and the remarks on the quantised situation in \cite{DG1}, Section 5). Since $\dotnabla(\barlambda)^F$ embeds in $\dotI(\barlambda)^F$ we must have $\dotnabla(\barlambda)=\dotI(\barlambda)$ and $\barlambda$ is injective for $\dotG(n)$.

\q Let $\mu$ be a maximal weight of $I(\lambda^0)$.  Now by \cite{DG1}, Lemma 3.1, $I(\lambda^0)\otimes \dotI(\barlambda)^F$ has $G(n)$-socle $L(\lambda)$ and so $I(\lambda^0)\otimes \dotI(\barlambda)^F$ embeds in $I(\lambda)$.  Thus  $\mu+l\barlambda$ is a weight of $I(\lambda)$ and so $I(\lambda)$ has a maximal weight $\tau$, say, such that $\tau\geq \mu+l\barlambda$.   But $I(\lambda)=\nabla(\lambda)$ has unique maximal weight $\lambda$ so that $\lambda\geq \tau\geq \mu+l\barlambda\geq \lambda^0+l\barlambda=\lambda$ and so   $\mu=\lambda^0$.  Hence $\lambda^0$ is a maximal weight of $I(\lambda^0)$ and so,  by Lemma 2.4,  $\lambda^0$ is injective.

\end{proof}


\begin{lemma} Let $\lambda$ be an injective partition and write $\lambda=\lambda^0+l\barlambda$, for partitions $\lambda^0,\barlambda$ with $\lambda^0$ being $l$-restricted. Then $\lambda^0$ is an $l$-core. 
\end{lemma}

\begin{proof}   By  the previous lemma we may assume $\lambda=\lambda^0$, i.e., that $\lambda$ is restricted.  Thus  $I(\lambda)$ is isomorphic to its contravariant dual, see e.g., \cite{D3}, 4.3,(2),(ii) , 4.3, (4) and (4.3), (ix).  Hence $I(\lambda)$ has simple head $L(\lambda)$. But $I(\lambda)=\nabla(\lambda)$ and $[\nabla(\lambda):L(\lambda)]=1$ so that in fact    $I(\lambda)=\nabla(\lambda)=L(\lambda)$.  Thus we get $[\mu:\lambda]=\delta_{\lambda,\mu}$ (the Kronecker delta) and $[\lambda:\tau]=\delta_{\lambda,\tau}$, for all partitions  $\mu,\tau$ with $|\mu|=|\tau|=|\lambda|$.  Hence $L(\lambda)$ is the only simple in its block (up to isomorphism), i.e., $\lambda$ is an  $l$-core.
\end{proof}

\q We introduce some  additional notation.  We set $\delta_0=0$ and $\delta_n=(n,n-1,\ldots,2,1)$, for $n\geq 1$.  We set $\sigma_0=0$ and 
$$\sigma_n=(l-1)\delta_n=(n(l-1),(n-1)(l-1),\ldots,2(l-1),(l-1))$$
for $n\geq 1$, so that $\sigma_n=(l-1)\delta_n$ for $n\geq 0$.  

\q   We call the partitions of the form   $\sigma_n$, for some $n\geq 0$, the   Steinberg partitions.  The justification for this is that in the classical case, with $K$ an algebraically closed field of characteristic $p>0$ the   restriction of  the ${\rm GL}_{n+1}(K)$-module $L(\sigma_n)$  to the special linear group ${\rm SL}_{n+1}(K)$ is the usual Steinberg module. 

\q Note that, since $\delta_n=\omega_n+\delta_{n-1}$ we have 
$$\sigma_n=\sigma_{n-1}+(l-1)\omega_n$$
for $n\geq 1$.

\begin{remark} {\rm{Suppose $n\geq 1$, $0\leq  a< l$ and  let $\mu$ be an injective partition of length at most $n$. We note that $\lambda=\sigma_{n-1}+a\omega_n+l\mu$ is injective. 
We have   that $\nabla(\sigma_{n-1}+a\omega_n)=\nabla(\sigma_{n-1})\otimes D_n^{\otimes a}$  is injective as a module for the first infinitesimal subgroup $G_1$ of $G(n)$ by \cite{D3}, Section 3.2, (12)  (and for example  
 \cite{RAG}, II, 10.2 Proposition in the classical case).   Hence by \cite{DG1}, Lemma 3.2(ii), and the remarks on the quantised situation in \cite{DG1}, Section 5,   we have $I(\sigma_{n-1}
+a\omega_n)=\nabla(\sigma_{n-1}+a\omega_n)$ and   $I(\lambda)=\nabla(\sigma_{n-1}+a\omega_n)\otimes \dotI(\mu)^F=\nabla(\sigma_{n-1}+a\omega_n)\otimes\dotnabla(\mu)^F$.  However, by  \cite{D3}, Section 3.2, (13) (and 
\cite{RAG}, II, 3.19 Proposition in the classical case) we have $\nabla(\lambda)=\nabla(\sigma_{n-1}+a\omega_n)\otimes\dotnabla(\mu)^F$ so that $\lambda=\sigma_{n-1}+a\omega_n+l\mu$ is injective.}} 

\end{remark}

\begin{remark} {\rm{Suppose   $\lambda=(\lambda_1,\ldots,\lambda_n)\in \Lambda^+(n)$  is an $l$-core and  $\lambda_n=l-1$. Then we have $\lambda=\sigma_n$. No doubt this is well known. We see it as follows. We may assume $n\geq 2$. Certainly $\lambda_{n-1}-\lambda_n<l$, for otherwise row $n-1$ of the diagram of $\lambda$ contains a skew $l$-hook. If $\lambda_{n-1}<2l-2$ then  there is a skew $l$-hook beginning at $(n-1,\lambda_{n-1})$ and ending at $(n,\lambda_{n-1}+2-l)$. Thus we have  $\lambda_{n-1}=2l-2$. Now $\mu=\lambda-(l-1)\omega_n$ is a $l$-core of length $n-1$ with last non-zero entry $l-1$.  Hence we can assume inductively that $\mu=\sigma_{n-1}$ and hence
$$\lambda=\sigma_{n-1}+(l-1)\omega_n=\sigma_n.$$}}

\end{remark}

\begin{lemma} If the partition  $\lambda$ is injective and ${\rm len}(\lambda^0)<{\rm len}(\barlambda)$ then $\lambda^0=\sigma_{n-1}$, where $n={\rm len}(\lambda)$.

\end{lemma}

\begin{proof} We consider $\mu=\lambda-\omega_n$.  Note that $\mu$ has length $n$ and $\mu_n$ is congruent to $-1$ modulo $l$. Hence, writing $\mu=\mu^0+l{\bar \mu}$, we have $\mu_n^0=l-1$.  Moreover, $\mu$ is injective, by Lemma 2.3, and so $\mu^0$ is injective by Lemma 2.6. Hence $\mu^0$ is a core, by Lemma 2.7 and $\mu^0=\sigma_n$, by Remark 2.9. Now we have 
$$\lambda=\mu+\omega_n=\sigma_n+\omega_n+l{\bar\mu}=\sigma_{n-1}+l({\barmu}+\omega_n)$$
and so $\lambda^0=\sigma_{n-1}$.

\end{proof}

\begin{lemma} Let  $\lambda$ be a partition of length $n$. If $\lambda$ is   injective then ${\rm len}(\barlambda)\leq {\rm len}(\lambda^0)+1$ and in case equality holds we have  $\lambda^0=\sigma_{n-1}$.

\end{lemma}

\begin{proof} If ${\rm len}(\barlambda)\geq {\rm len}(\lambda^0)+1$ then ${\rm len}(\barlambda)>{\rm len}(\lambda^0)$ so that $n={\rm len}(\barlambda)$ and ${\rm len}(\lambda^0)<n$. Hence $\lambda^0=\sigma_{n-1}$ by Lemma 2.10 and ${\rm len}(\barlambda)={\rm len}(\lambda^0)+1$.

\end{proof}

\begin{lemma} Suppose that the partition $\lambda$ satisfies ${\rm len}( \lambda^0)={\rm len}(\lambda)$ and $\lambda^0$ is an $l$-core. If $\lambda-\omega_n$ is injective, where $n$ is the length of $\lambda$,  then so is $\lambda$. 

\end{lemma}

\begin{proof}  Suppose $\mu$ is a partition such that $\mu>\lambda$ and $[\mu:\lambda]\neq 0$.  Then $\mu$ also has core $\lambda^0$ and so $\mu$ has length $n$. Thus we may write $\mu=\tau+\omega_n$, for some partition $\tau$.  But then
$$[\mu:\lambda]=[\tau+\omega_n:\lambda]=[\tau:\lambda-\omega_n]=0.$$
Thus no such partition $\mu$ exists and $\lambda$ is injective.

\end{proof}

\begin{definition} We define the Steinberg index  $\stind_l(\lambda)$ relative to $l$  of a partition $\lambda$  to be $0$ if $\lambda_1-\lambda_2\neq l-1$ and otherwise to be $m>0$ if $\lambda_i-\lambda_{i+1}=l-1$ for $1\leq i\leq m$ and $\lambda_{m+1}-\lambda_{m+2}\neq l-1$. (Thus for example $\stind(\sigma_n)=n$, for $n\geq 0$).

\end{definition}

\begin{proposition} Let $\lambda$ be a partition written  $\lambda=\lambda^0+l\barlambda$ in standard form.  Then $\lambda$ is injective if and only if $\lambda^0$ is an $l$-core, $\barlambda$ is injective and ${\rm len}(\barlambda)\leq \stind_l(\lambda^0)+1$.

\end{proposition}

\begin{proof}    Let $n={\rm len}(\lambda)$.

\q We first  suppose $\lambda$ is injective.   Then $\barlambda$ is injective, by Lemma 2.6 and $\lambda^0$ is an $l$-core, by Lemma 2.7. We claim that also ${\rm len}(\barlambda)\leq \stind_l(\lambda^0)+1$.

\q We know that 
${\rm len}(\bar\lambda)\leq {\rm len}(\lambda^0)+1$, by Lemma 2.11. Moreover, if ${\rm len}(\bar\lambda)= {\rm len}(\lambda^0)+1$ then $\lambda^0=\sigma_{n-1}$ and so $\stind_l(\lambda^0)=n-1$, ${\rm len}(\barlambda)=n$, by Lemma 2.11, and the desired conclusion holds.  Now suppose that the 
 claim is false and that $\lambda$ is an injective partition of minimal degree  for which it fails.  Thus we have ${\rm len}(\bar\lambda)\leq {\rm len}(\lambda^0)=n$ by the case already considered.  Thus we must have that $n\geq 2$ and that   $\stind_l(\lambda^0)=m$, say, is at most $n-2$.  Now $\mu=\lambda-\omega_n=(\lambda^0-\omega_n)+l\barlambda$ is injective, by Lemma 2.3. But we have $\stind_l(\lambda^0-\omega_n)=\stind_l(\lambda^0)$ and so, by minimality
 $${\rm len}(\bar\lambda)\leq \stind_l(\lambda^0-\omega_n)+1=\stind_l(\lambda^0)+1$$
  and the claim is proved.

\q We now suppose that $\barlambda$ is injective, that $\lambda^0$ is an $l$-core and  ${\rm len}(\barlambda)\leq \stind_l(\lambda^0)+1$. We show that $\lambda$ is injective by induction on the degree of $\lambda$.

\q If the  Steinberg index of $\lambda^0$ is $n$ then $\lambda^0=\sigma_n$ and $\lambda$ is injective by Remark 2.8. 

\q If  the Steinberg index of $\lambda^0$ is $n-1$ then $\lambda^0$ has the form $\sigma_{n-1}+a\omega_n$, for some $0\leq a<l$ and this case is also covered by Remark 2.8. 

\q Thus we may assume that  $\stind_l(\lambda^0)<n-1$.   Then  ${\rm len}(\barlambda)<n$ so that ${\rm len}(\lambda^0)=n$. By Lemma 2.12 it is enough to show that $\lambda-\omega_n$ is injective.    But we have
$$\lambda-\omega_n=(\lambda^0-\omega_n)+p\barlambda$$
and so $\stind_l(\lambda^0-\omega_n)=\stind_l(\lambda^0)$ and we are done by induction.

\end{proof}

This solves the problem of determining which partitions are injective for $G(n)$. We separate out the cases.

\begin{corollary} Suppose $K$ has characteristic $0$. Then a partition $\lambda$ is injective for $G(n)$ if and only if  $\lambda^0$ is an $l$-core and ${\rm len}(\barlambda)\leq \stind_l(\lambda^0)+1$.

\end{corollary}

\begin{proof}  In this case all $\dotG(n)$-modules are completely reducible so that $\barlambda$ is injective for $\dotG(n)$ and the result follows from Proposition 2.14.

\end{proof}

It remains to consider the case in which $K$ has characteristic $p>0$.  A partition $\lambda$ has  unique  base $p$ expansion $\lambda=\sum_{i\geq 0} p^i\lambda^i$, where each $\lambda^i$ is a $p$-restricted partition. The final results follow immediately from Proposition 2.14.

\begin{corollary} Suppose $K$ has characteristic $p>0$ and $q=1$. Let $\lambda$ be a partition with base $p$ expansion $\lambda=\sum_{i\geq 0} p^i\lambda^i$.  Then  $\lambda$ is injective if and only if each $\lambda^i$ is a $p$-core and ${\rm len}(\lambda^j)\leq \stind_p(\lambda^i)+1$, for all $0\leq i<j$,

\end{corollary}

\begin{corollary} Suppose $K$ has characteristic $p>0$ and $q$ is an $l$th root of unity, with $l>1$.  Let $\lambda$ be a partition written in standard form $\lambda=\lambda^0+l\barlambda$ and write $\barlambda$ in its base $p$ expansion $\barlambda=\sum_{i\geq 0} p^i\barlambda^i$.  Then  $\lambda$ is injective if and only if   $\lambda^0$ is an $l$-core and $\barlambda^i$ is a $p$-core for each $i\geq 0$  and we have ${\rm len}(\barlambda^j)\leq \stind_l(\lambda^0)+1$  for all $j> 0$ and ${\rm len}(\barlambda^j)\leq \stind_p(\barlambda^i)+1$, for all $0<i<j$.

\end{corollary}

\begin{examples}
{\rm We give here one example of a partition that is injective and one of a partition that is not for the case in which  $K$ is a field of characteristic $3$ and $q$ is a primitive $4$th root of unity.  We  test these partitions using Corollary 2.17.\\

(i) Consider first the partition $\lambda=(20,9,6)$. We write $\lambda$ in the standard form $\lambda=(8,5,2)+4(3,1,1)$. We have that $(8,5,2)$ is a $4$-core and the partition $(3,1,1)$ is a $3$-core. Moreover $\stind_4(8,5,2)=2$ and since $(3,1,1)$ has length $3$ we get that $\lambda=(20,9,6)$ is an injective partition.\\

(ii) Consider now the partition $\mu=(17,6,4)$. We write $\mu$ in the standard form $(5,2)+4(3,1,1)$. We have that $(5,2)$ is a $4$-core and the partition $(3,1,1)$ is a $3$-core. Here, $\stind_4(5,2)=1$ and since $(3,1,1)$ has length $3$ we get that $\lambda=(17,6,4)$ is not an injective partition.}

\end{examples}

\q With these results in hand  we can now describe which Weyl modules are projective in the category of the polynomial $G(n)$-modules.

\begin{corollary}   
Let $\lambda\in \Lambda^+(n,r)$. Then the Weyl module $\Delta(\lambda)$ is a projective polynomial $G(n)$-module, and so $\Delta(\lambda)=P(\lambda)$, if and only if $\lambda$ is an injective partition.
\end{corollary}

\begin{proof}

This is clear since $\Delta(\lambda)$ is the contravariant dual of $\nabla(\lambda)$ and $P(\lambda)$ the contravariant dual of $I(\lambda)$ (see \cite{D3}, Section 4.1).

\end{proof}

\q We end this section by pointing out that   our criterion for describing injective partitions also describes the set of regular partitions satisfying Carter's criterion.   We now describe the \lq\lq $(l,p)$-adic" valuation of a positive integer, where $l\geq 2$ is a positive integer and $p$ is a prime.  For a positive integer $r$ we define $\nu_{l,p}(r)$ to be $0$ if $l$ does not divide $r$. If $l$ divides $r$ then $\nu_{l,p}$ is $1+\nu_p(r/l)$ (where $\nu_p$ denotes the $p$-adic valuation on non-zero integers).

\q For the rest of this section $p$ denotes the characteristic of our base field $K$ and $l$ is the smallest integer such that $1+q+\cdots+q^{l-1}=0$.  As above we shall say that a partition $\lambda$ is injective if  the $G(n)$-module $\nabla(\lambda)$ is polynomially injective.  If we wish to emphasise the roles of $l$ and $p$ we shall say that $\lambda$ is $(l,p)$-injective.


\q    We write $[\lambda]$ for the diagram of a partition $\lambda=(\lambda_1,\lambda_2,\ldots)$.   For $(a,b)\in [\lambda]$ we write $h_{(a,b)}^\lambda$ for the corresponding hook length (i.e., $\lambda_a+\lambda_b'-a-b+1$, where $\lambda'=(\lambda_1',\lambda_2',\ldots)$ is the transpose of $\lambda$).

\q Recall that a partition $\lambda$ is $l$-regular if there is no $t\geq 0$ with $\lambda_{t+1}=\lambda_{t+2}=\cdots=\lambda_{t+l}>0$.

\q We shall say that an  $l$-regular partition $\lambda$ is a Carter partition, or that $\lambda$ satisfies Carter's criterion, if $\nu_{l,p}$ is constant on the hook lengths corresponding to nodes in each column of $[\lambda]$, i.e., if $\nu_{l,p}(h_{(a,b)}^\lambda)$ is independent of $a$, for $(a,b)\in [\lambda]$. We shall also say  that a Carter partition is a Carter partition relative to $(l,p)$, when  we wish to emphasise the roles of $l$ and $p$.

\bs

\q We shall use repeatedly the well-known (and easy to prove) fact that a partition $\lambda$ is an $l$-core if and only if no hook length $h^\lambda_{(a,b)}$ is divisible by $l$.

\begin{proposition}  A partition $\lambda$ is injective if and only if it is an $l$-regular Carter partition.

\end{proposition}

\begin{proof}  As usual  we write a partition $\lambda$ in the form  $\lambda^0+l\barlambda$,  for partitions $\lambda^0$, $\barlambda$  with $\lambda^0$ being $l$-restricted. 

\q   We first show that an injective partition $\lambda$ is an $l$-regular Carter partition.  We argue by induction on degree. We may assume that $\lambda\neq 0$ and that the result holds for partitions of smaller degree. 
    By Lemma 2.7, $\lambda^0$ is an $l$-core. Now if, for some $t\geq 0$, we have $\lambda_{t+1}=\lambda_{t+2}=\cdots=\lambda_{t+l}>0$, then we get $\lambda_{t+1}^0\equiv \cdots \equiv \lambda^0_{t+l}$ mod $l$, and since $\lambda^0$ is $l$-restricted,  $\lambda_{t+1}^0 =\cdots = \lambda^0_{t+l}$. If $\lambda_{t+1}^0>0$ then $\lambda^0$ contains a removable skew $l$-hook, which is impossible since $\lambda^0$ is an $l$-core. Thus we have $\lambda_{t+1}^0=\cdots=\lambda^0_{t+l}=0$ and so $\lambda^0$ has length at most $t$ and $\barlambda_{t+1}=\cdots=\barlambda_{t+l}>0$ and so the length of $\barlambda$ is at least $t+l\geq t+2$,  and this is impossible since the length of $\barlambda$ is at most one more than the length of $\lambda^0$, by  Proposition 2.11. Hence $\lambda$ is $l$-regular. 

   \q Let $n$ be the length of $\lambda$. Let   $\mu=\lambda-\omega_n$.  Then $\mu$ is injective and hence, by the inductive hypothesis, an $l$-regular Carter partition.   Moreover, for $b>1$ and $(a,b)\in [\lambda]$ we have $h_{(a,b)}^\lambda=h_{(a,b-1)}^\mu$, and hence $\nu_{l,p}(h_{(a,b)}^\lambda)=\nu_{l,p}(h_{(a,b-1)}^\mu)$ and this is independent of $a$. It remains to check the case  $b=1$. 
   
   \q Assume first that $\lambda^0$ has length $n$. Then we have 
   $$h_{(a,1)}^\lambda=\lambda^0_a+l\barlambda_a+n-a=h_{(a,1)}^{\lambda^0}+l\barlambda_a.$$
    Since $\lambda^0$ is an $l$ core this is not a multiple of $l$ and $\nu_{l,p}$ takes constant value $0$ on $h_{(a,1)}^\lambda$, $1\leq a \leq n$.

\q Now assume that $\lambda^0$ has length less then $n$. Then its length is $n-1$ and indeed, by Lemma 2.11, we have  $\lambda^0=\sigma_{n-1}$.  Thus we have $\lambda^0_a=(n-a)(l-1)$ and 
$$h_{(a,1)}^\lambda=\lambda_a^0+l\barlambda_a+n-a=(n-a)(l-1)+l\barlambda_a+(n-a)=l(\barlambda_a+n-a)$$
and hence $\nu_{l,p}(h_{(a,1)}^\lambda)=\nu_p(h_{(a,1)}^\barlambda)+1$, for $1 \leq a \leq  n$, and this is independent of $a$.  Hence $\lambda$ is an $l$-regular Carter partition.

\q Now assume for a contradiction that there is a $l$-regular Carter partition that is not injective and let $\lambda$ be such a partition of smallest possible degree.   Let $n$ be the length of $\lambda$.  We first show that $\lambda^0$ is an $l$-core.  

\q  Assume first that  $\lambda^0$ has length $n$. Note that $\lambda-\omega_n=(\lambda^0-\omega_n)+l\barlambda$ is an $l$-regular Carter partition, and hence injective. Hence $\lambda^0-\omega_n$ is an $l$-core, by the minimality assumption,  and hence  $l$ does not divide $h^{\lambda^0}_{(a,b)}=h^{\lambda^0-\omega_n}_{(a,b-1)}$, for $(a,b)\in [\lambda^0]$ and $b>1$.  We now consider $h^\lambda_{(a,1)}$, for $1\leq a\leq n$. We have $\nu_{l,p}(h^\lambda_{(n,1)})=\nu_{l,p}(\lambda_n^0+l\barlambda)$, and this is $0$  as $0<\lambda^0_n<l$. Hence we have $\nu_{l,p}(h^\lambda_{(a,1)})=0$, i.e., $\lambda_a^0+l\barlambda_a+n-a$ is not divisible by $l$,  i.e., $\lambda_a^0+n-a=h^{\lambda^0}_{(a,1)}$ is not divisible by $l$,  for all $1\leq a\leq n$. Thus $\lambda^0$ is an  $l$-core. 

\q Now assume that $n$ is greater than the length of $\lambda^0$. Then $\lambda-l\omega_n=\lambda^0+l(\barlambda-\omega_n)$ is Carter and therefore  injective and so $\lambda^0$ is a core, by the minimality assumption.

\q  Suppose once more  that $\lambda^0$ has length $n$. Since  $\lambda-\omega_n$ is Carter, and hence injective, we get that $\lambda$ is injective from Lemma 2.2.  Hence we may assume that $\lambda^0$ has length less than $n$.

\q Consider the  partition $\mu=(\lambda_2,\lambda_3,\ldots)$  whose diagram is obtained by removing the first  row of $[\lambda]$.  Then $\mu$ is an $l$-regular Carter partition and hence injective. We write $\mu=\mu^0+l\barmu$,  as usual.  Then the length of $\mu^0$ is less than the length of $\mu$,  which is  $n-1$.   This implies, by Lemma 2.11 that $\mu^0=\sigma_{n-2}$, i.e., we have $\lambda_a^0=(n-a)(l-1)$,   for $2\leq a\leq n$.  Now 
$$h_{(2,1)}^\lambda=\lambda_2^0+l\barlambda_2+(n-2)=l(n-2)+l\barlambda_2$$
which is divisible by $l$.  Thus we know that $h_{(1,1)}^\lambda$ is also divisible by $l$, i.e., $l$ divides
 $\lambda_1+n-1=(\lambda_1-\lambda_2)+1+h^\lambda_{(2,1)}$. Hence $\lambda_1-\lambda_2+1$ is also divisible by $l$ and therefore  $\lambda_1^0-\lambda_2^0+1$ is divisible by $l$. Thus we have $\lambda_1^0-\lambda^0_2=l-1$, $\lambda^0_1=(n-1)(l-1)$ and $\lambda^0=\sigma_{n-1}$.  By  Proposition 2.14 we will be done if we show that $\barlambda$ is a $p$-regular Carter partition (and hence injective).

\q Now $\lambda-l\omega_n=
\lambda^0+l(\barlambda-\omega_n)$ is an $l$-regular Carter partition, and so  injective and hence, by Proposition 2.14,  $\barlambda-\omega_n$ is Carter. Thus $\barlambda$ is $p$-regular and,  moreover,  $\nu_{l,p}(h^\lambda_{(a,b)})$ is independent of $a$, for $b>1$, $(a,b)\in [\lambda]$.  It remains to prove that $\nu_p$ is constant on $h^\barlambda_{(a,1)}$, for $1\leq a\leq n$. However, we have 
\begin{align*}\nu_{l,p}(h_{(a,1)}^\lambda))&=\nu_{l,p}(\lambda_a+n-a)=\nu_{l,p}((n-a)(l-1)+l\barlambda_a+n-a)\cr
&=\nu_{l,p}(l(\barlambda_a+n-a))=1+\nu_p(h^\barlambda_{(a,1)})
\end{align*}
and $\nu_{l,p}(h_{(a,1)}^\lambda)$ is independent of $a$ and so $\nu_p(h^\barlambda_{(a,1)})$ is independent of $a$, and we are done.

\end{proof}

\begin{remark}  {\rm{It is convenient to treat the case in which $K$ has characteristic $0$ separately. In that case we say that a partition $\lambda$ is a Carter partition relative to $(l,0)$, if for $1\leq b\leq \lambda_1$ either  $h_{a,b}$ is divisible by $l$ for all $1\leq a\leq \lambda_b'$, or no  $h_{a,b}$ is divisible by $l$ for all $1\leq a\leq \lambda_b'$.  Then one may easily check that $\lambda$ is injective relative to $(l,0)$ if and only if it is a Carter partition relative to $(l,0)$.}} 
\end{remark}

\begin{remark}  {\rm{The irreducible Weyl modules for the general linear group were determined by Jantzen, \cite{JantzenThesis}, Teil II, 8, Satz 9, using the Jantzen sum formula. This formula was used, in the quantised case, by Mathas, Theorem 5.39,  to determine the irreducible Weyl modules for the quantised Schur algebra $S_q(n,n)$. 
We here indicate the connection between this and our point of view.  Let $n$ be a positive integer and $r\geq 0$. For $\lambda\in \Lambda^+(n,r)$ we write $T_n(\lambda)$ for the corresponding tilting module for $G(n)$, as in \cite{D3}.  

\q We now fix a partition $\lambda$, let $r$ be its degree.   Assume that $\lambda$ is injective. Then 
Remark 2.2,  we have  $[\mu:\lambda]=0$ for every partition $\mu>\lambda$.  Let $\tau\in \Lambda(r,r)$. Then we have 
$(T_r(\lambda'):\nabla_r(\tau'))=[\tau:\lambda]$, by \cite{D3},  4.2, (14).  It follows that we have that $(T_r(\lambda'):\nabla_r(\tau'))$ is one if $\tau'=\lambda'$ and $0$ otherwise, i.e., we have $T_r(\lambda')=\nabla_r(\lambda')=\Delta_r(\lambda')=L_r(\lambda')$.  Reversing the steps we see that if $\lambda$ is a partition of degree $r$ such that  $\Delta_r(\lambda')=L_r(\lambda')$ then $\lambda$ is injective.  Hence a partition of degree $r$ is injective if and only if the Weyl module $\Delta_r(\lambda')$ for $G(r)$, equivalently for the Schur algebra $S_q(r,r)$, is irreducible. Such partitions are described, in terms of Carter's condition, by Mathas, \cite{MathasBook}, Theorem 5.39. Thus Mathas's Theorem may be used to describe the injective partitions.  Conversely, one may use the above to give an alternative proof of Mathas's Theorem.}}

\end{remark}

  \bs\bs\bs


\section{Specht modules which are Young modules}

\q In this section we express our main result  in terms of the representation theory of Hecke algebras and so  determine the partitions  $\lambda$ for which  the corresponding Specht module $\Sp(\lambda)$ is a Young module, for $l\geq3$.   We give an example point out that in the case $l=2$ not all such $\lambda$ are injective partitions. We intend to study this phenomenon in more detail in a subsequent work.

\begin{definition}   We call  a partition $\lambda$ of degree $r$ a  Young partition for $H(r)$,   or just a Young partition, if the corresponding Specht  module $\Sp(\lambda)$,  is a Young module.

\end{definition}

\begin{proposition}    

(i) Every injective partition is a Young partition.

Assume now that $l\geq 3$ and let $\lambda$ be a partitions of degree $r$. 

(ii) If $\lambda$ is a Young partition then $\lambda$ is injective and  $\Sp(\lambda)=Y(\lambda)$.

\end{proposition}

\begin{proof}  (i) holds by Proposition 1.3.1. 

(ii)  For $l\geq 4$ the result follows from the fact that, for $n\geq r$  and a polynomial $G(n)$-module $X$ of degree $r$ with a good filtration, the Specht module multiplicities in $fX$ agree with the corresponding $\nabla$-module multiplicities in $X$ (see the remarks before Proposition 1.3.3).  However, we here give a different argument which also covers the case $l=3$. 

\q So let $\lambda$ be a Young partition of degree $r$, say, and let $n\geq r$. Thus we have $\Sp(\lambda)=Y(\mu)$, for some partition $\mu$, of degree $r$. Assume, for a contradiction, that $\lambda$ is not injective. Then we have that $\mu$ is not injective, since for otherwise $\nabla(\mu)=I(\mu)$ and so $\Sp(\mu)=Y(\mu)$, and then $\Sp(\lambda)=Y(\mu)=\Sp(\mu)$ which gives  $\lambda=\mu$ by Lemma 1.3.2(ii). 

\q Since $\mu$ is not injective, we have a short exact sequence 
$$0\rightarrow \nabla(\mu)\rightarrow I(\mu)\rightarrow X\rightarrow0$$
where  $X$ is a non-zero $G(n)$-module with a good filtration with sections of the form $\nabla(\tau)$, with $\tau\in \Lambda^+(n,r)$, $\tau>\mu$. Applying $f$ we obtain a short exact sequence 
 $$0\rightarrow \Sp(\mu)\rightarrow \Sp(\lambda)\rightarrow fX\rightarrow0$$
 where $fX$ is  a non-zero  $H(r)$-module with a filtration whose sections have the form  $\Sp(\tau)$, with $\tau\in \Lambda^+(n,r)$, $\tau>\mu$.  Since $\Hom_{H(r)}(\Sp(\mu),\Sp(\lambda))\neq 0$ we have $\mu\geq \lambda$, by Proposition 1.3.2 (ii).  Since $\Hom_{H(r)}(\Sp(\lambda),fX)\neq 0$ we have, by left exactness, $\Hom_{H(r)}(\Sp(\lambda),\Sp(\tau))\neq 0$, for some section $\Sp(\tau)$ of $fX$, and hence by Proposition 1.3.2(ii), $\lambda\geq \tau$, for some $\tau>\mu$.  But now we have $\lambda\geq \tau>\mu\geq \lambda$, which is impossible. Therefore $\lambda$ is  injective and $\Sp(\lambda)=Y(\lambda)$.
\end{proof}

\begin{remark} {\rm{We note that in \cite{Hem}, Proposition 1.1,  an argument is given to establish the above result in the classical case, in characteristic  $p\geq 3$. We needed to adopt the above, somewhat different, strategy  since Hemmer's argument relies on the result \cite{HN}, 3.4.2,  valid  for $p\geq 5$.}}
\end{remark}


\begin{example} \rm  We give here an example to point out that Proposition 3.2(ii) does not in general hold for $l=2$ and we may have $\Sp(\lambda)\neq Y(\lambda)$ but $\Sp(\lambda)=Y(\mu)$, for distinct partitions $\lambda$ and $\mu$. We take $q=1$ and take  $K$ to be a field of characteristic   $2$. We will assume of the reader some familiarity   with the  description of the basis of the Specht modules via  polytabloids and of the usual basis  of the permutation modules. More details   can be found in \cite{James}.

\q We consider the partition $\lambda=(3,1,1)$. Writing $\lambda$ in the usual form $\lambda=\lambda^0+2\bar\lambda$ we have  $(3,1,1)=(1,1,1)+2(1)$ and by Corollary 2.16  this is not an injective partition,  since $\lambda^0=(1,1,1)$ is not a $2$-core. However we will show that $(3,1,1)$ is a Young partition and in fact that $\Sp(3,1,1)=Y(3,2)$. In order to do this we study first the injective module $I(3,2)$ for $G(n)$, with $n\geq 5$. It is easy to see that  we have a short exact sequence
 
$$0\rightarrow \nabla(3,2)\rightarrow I(3,2)\rightarrow \nabla(5)\rightarrow0.$$

Applying  the Schur functor we get

 $$0\rightarrow \Sp(3,2)\rightarrow Y(3,2)\rightarrow \Sp(5)\rightarrow0.$$
 
 Moreover  the permutation module $M(3,2)$ is the direct sum $M(3,2)\cong Y(3,2)\oplus \Sp(4,1)$,  since $(4,1)$ and $(3,2)$ have different cores. Also, we  have  $\dim\Sp(3,1,1)=\dim Y(3,2)=6$.  We identify  the Specht module $\Sp(3,1,1)$ with  a submodule of  $M(3,2)$. The sets $\{i,j\}$ with $1\leq i\neq j\leq5$ form a basis of $M(3,2)$ and the polytabloids  $e_T$,  with $T$ a standard  tableau of shape $(3,1,1)$,  form a basis of $\Sp(3,1,1)$. We define  $\phi: \Sp(3,1,1)\rightarrow M(3,2)$ to be the $K$-linear map sending   the poytabloid corresponding to the standard tableau with  entries $1<i<j$ in its first column to the sum $\{1i\}+\{1j\}+\{ij\}$. It is easy to see that $\phi$ commutes with the action of the symmetric group $\Sym_5$  and that it is injective. Since the cores of $(3,1,1)$ and $(4,1)$ are different 
 $\phi$ gives an embedding of  $\Sp(3,1,1)$ into the Young module $Y(3,2)$. Moreover, $\dim Y(3,2)=\dim\Sp(3,1,1)$ so that $\Sp(3,1,1)\cong Y(3,2)$.

\end{example}

\section*{Acknowledgement}

The second author gratefully acknowledges the financial  support of EPSRC Grant EP/L005328/1.

\bs\bs\bs


\end{document}